\documentclass[11pt,reqno]{amsart}

\usepackage{url}
\usepackage[top=1in, bottom=1in, left=0.8in, right=0.8in]{geometry}

\usepackage[dvipsnames]{xcolor}

\usepackage{framed,subcaption} 
\usepackage{amsthm,amsmath,amssymb,graphicx}
\usepackage{hyperref}

\newcommand{\citep}{\cite} 
\newcommand{\cf}[0]{\textit{cf}.\ } 
\renewcommand{\emph}[1]{\textit{#1}}

\usepackage{diagbox}

\DeclareMathOperator{\Log}{Log} 
 
\DeclareMathOperator{\erf}{erf} 

\renewcommand{\Re}{\operatorname{Re}}
\renewcommand{\Im}{\operatorname{Im}}

\theoremstyle{plain}
\newtheorem{thm}{Theorem}[section]
\newtheorem{theorem}[thm]{Theorem}

\newtheorem{lemma}[thm]{Lemma}

\theoremstyle{definition}

\newtheorem*{unremark}{Remark}

\begin{document} 
\title[Integral Transforms and OGF-to-EGF Conversion Formulas]{
       A Short Note on Integral Transformations and Conversion Formulas for 
       Sequence Generating Functions} 
\markright{Integral Transformations of Generating Functions}
\author{Maxie D. Schmidt} 

\address{School of Mathematics \\ 
        Georgia Institute of Technology \\ 
        Atlanta, GA 30318 \\ 
        USA} 
\email{maxieds@gmail.com, mschmidt34@gatech.edu} 

\allowdisplaybreaks 

\date{\today}

\begin{abstract}
The purpose of this note is to provide an expository introduction to some more 
curious integral formulas and transformations involving generating functions. We seek to 
generalize these results and integral representations which effectively provide a 
mechanism for converting between a sequence's 
\emph{ordinary} and \emph{exponential generating function} 
(OGF and EGF, respectively) and vice versa. 
The Laplace transform provides an integral formula for the EGF-to-OGF transformation, 
where the reverse OGF-to-EGF operation requires more careful integration techniques. 
We prove two variants of the OGF-to-EGF transformation integrals from the 
Hankel loop contour for the reciprocal gamma function and from Fourier series expansions 
of integral representations for the Hadamard product of two generating functions, 
respectively. 
We also suggest several generalizations of these integral formulas and provide new 
examples along the way. 
\end{abstract}

\maketitle

\section{Introduction} 

\subsection{Definitions} 

Given a sequence $\{f_n\}_{n \geq 0}$, we adopt the notation for the respective 
\emph{ordinary generating function} (OGF), $F(z)$, and 
\emph{exponential generating function} (EGF), 
$\widehat{F}(z)$, of the sequence in some formal indeterminate parameter 
$z \in \mathbb{C}$: 
\begin{align} 
F(z) & = \sum_{n \geq 0} f_n z^n \\ 
\notag 
\widehat{F}(z) & = \sum_{n \geq 0} \frac{f_n}{n!} z^n. 
\end{align} 
Notice that we can always construct these functions over any sequence 
$\{f_n\}_{n \in \mathbb{N}}$ and formally perform operations on these functions 
within the ring of formal power series in $z$ 
without any considerations on the constraints imposed by the convergence of the 
underlying series a a complex function of $z$. If we assume that the 
respective series for $F(z)$ or $\widehat{F}(z)$ is analytic, or converges absolutely, 
for all $z \in \mathbb{C}$ with $0 < |z| < \sigma_f$, then we can apply 
complex function theory to these sequence generating functions and treat them as 
analytic functions of $z$ on this region.  

We can precisely define the form of an 
\emph{integral transformation} (in one variable) as \cite[\S 1.4]{INTTFSANDAPPS} 
\begin{equation} 
\label{eqn_GenIntTransformDef_v1} 
\mathcal{I}[f(x)](k) := \int_a^b \mathcal{K}(x, k) f(x) dx, 
\end{equation}
for $-\infty \leq b < a \leq +\infty$ and where the function 
$\mathcal{K}: \mathbb{R} \times \mathbb{C} \rightarrow \mathbb{C}$ is called 
the \emph{kernel} of the transformation. When the function $f$ which we operate on in the 
formula given by the last equation corresponds to an OGF or EGF of a sequence with which 
we are concerned in applications, we consider integrals of the form in 
\eqref{eqn_GenIntTransformDef_v1} to be so-called \emph{generating function transformations}. 
Such generating function transformations are employed to transform the ordinary power series 
of the target generating function for one sequence into the form of a generating function 
which enumerates another sequence we are interested in studying. 

Generating function transformations form a useful combinatorial and analytic method 
(depending on perspective) which can be combined and employed to study new sequences of many forms. 
Our focus in this article is to motivate the constructions of generating function transformations 
as meaningful and indispensable tools in enumerative combinatorics, combinatorial number theory, and 
in the theory of partitions, among other fields where such applications live. 
The particular modus operandi within this article shows the evolution of integral transforms for the 
reciprocal gamma function, and its multi-factorial integer sequence special cases, as a motivating 
method for enumerating several types of special sequences and series which we will consider in the 
next sections. 

The references \cite{GKP,ADVCOMB,ECV2} 
provide a much broader sense of the applications 
of generating function techniques in general to those readers who are not familiar with this 
topic as a means for sequence enumeration. A comprehensive array of analytic and experimental 
techniques in the theory of integral transformations is also treated in the references 
\cite{INTTFSANDAPPS,II}. We focus on only a comparatively few concrete examples of 
integral and sequence transformations in the next subsections with hopes 
to motivate our primary results proved in this article from this perspective. 
We hope that the discussion of these techniques in this short note provide motivation and 
useful applications to readers in a broader range of mathematical areas. 

\subsection{From hobby to short note: OGF-to-EGF conversion formulas} 
\label{Section_StdIntFormulas_Constructions} 

A time consuming hobby that the author assumes from time to time is rediscovering 
old and unusual identities in mathematics textbooks-- particularly in the areas of 
combinatorics and discrete mathematics. 
Favorite books to search include Comtet's \emph{Advanced Combinatorics} and the 
exercises and their solutions found in \emph{Concrete Mathematics} by Graham, Knuth and 
Patashnik. One curious and interesting conversion operation discussed in the exercises to 
Chapter 7 of the latter book involves a pair of integral formulas for converting an 
arbitrary sequence OGF into its EGF and vice versa provided the resulting integral is 
suitably convergent. The exercise listed in \emph{Concrete Mathematics} suggests the 
second form of the operation. Namely, that of converting a sequence EGF into its OGF. 

In this direction, we have an easy conversion integral for 
converting from the EGF of a sequence $\{f_n\}_{n \geq 0}$, denoted by 
$\widehat{F}(z)$, and its corresponding OGF, denoted by $F(z)$, given by the 
\emph{Laplace-Borel transform} \cite[\S B.14]{ANALYTIC-COMB}: 
\[
\mathcal{L}[\widehat{F}](z) = F(z) = \int_0^{\infty} \widehat{F}(tz) e^{-t} dt. 
\] 
Other integral formulas for conversions between specified generating function 
``\emph{types}'' can be constructed similarly as well 
(see Section \ref{subSection_Intro_IntTF_Examples}). 
The key facets in constructing these semi-standard, or at least known, 
conversion integrals is in applying a termwise series operation which generates a 
factor, or reciprocal factor, of the gamma function $\Gamma(z+1)$ when 
$z \in \mathbb{N}$. 
The corresponding ``\emph{reversion}'' operation of converting from a sequence's 
OGF to its EGF requires a more careful treatment of the properties of the 
reciprocal gamma function, $1 / \Gamma(z+1)$, and the construction of integral formulas 
which generate it for $z \in \mathbb{N}$ involving the \emph{Hankel loop contour} 
described in Section \ref{Section_HankelLoop}. 

That being said, Graham, Knuth and Patshnik already suggest a curious 
``\emph{known}'' integral formula 
for performing this corresponding OGF-to-EGF conversion operation of the following form 
\cite[p.\ 566]{GKP}: 
\begin{align} 
\label{eqn_CMATH_OGF2EGF_int_formula} 
\widehat{F}(z) = \frac{1}{2\pi} \int_{-\pi}^{\pi} F\left(z e^{-\imath t}\right) 
     e^{e^{\imath t}} dt. 
\end{align} 
The statement of this result is given without proof in the identity-full appendix 
section of the textbook. When first (re)-discovered many years back, the author 
assumed that the motivation for this integral transformation must correspond to the 
non-zero paths of a complex contour integral for the reciprocal gamma function. 
For many years the precise formulation of a proof of this termwise integral formula and 
its generalization to enumerating terms of reciprocal generalized multi-factorial functions, such as $1 / (2n-1)!!$, remained a mystery and curiosity of periodic interest to the author. 
In the summer of 2017, the author finally decided to formally inquire about the proof and 
possible generalizations in an online mathematics forum.  
The question went unanswered for over a year until by chance the author stumbled onto 
a Fourier series identity which finally motivated a rigorous proof of the formula in 
\eqref{eqn_CMATH_OGF2EGF_int_formula}. 
This note explains this proof and derives another integral formula for this operation of 
OGF-to-EGF inversion based on the Hankel loop contour. 
The preparation of this article is intended to be expository in nature in the hope of 
inspiring the creativity of more researchers towards developing related integral 
transformations of sequence generating functions. 

\subsection{Examples: Integral transformations of a sequence generating function} 
\label{subSection_Intro_IntTF_Examples} 

Integral transformations are a powerful and convenient formal and analytic tool 
which are used to study sequences and their properties. 
Moreover, they are easy to parse and apply in many contexts with only basic knowledge of 
infinitesimal calculus making them easy-to-understand operations which we can apply to 
sequence generating functions. 
The author is an enthusiast for particularly pretty or interesting integral representations 
(\cf \cite{II,INTSERIES-TABLES}) and has taken a special research interest in finding 
integral formulas of the ordinary generating function of sequence which transform the 
series into another generating function enumerating a modified special sequence. 

One notable example of such an integral transformation given in 
\cite[\S 2]{EXPLICIT-EVAL-ESUMS} allows us to construct 
generalized polylogarithm-like and Dirichlet-like series over any 
prescribed sequence in the following forms for integers $r \geq 1$: 
\begin{align}
\sum_{n \geq 0} \frac{f_n}{(n+1)^r} z^n & = 
     \frac{(-1)^{r-1}}{(r-1)!} \int_0^1 \log^{r-1}(t) F(tz) dt \\ 
\notag 
     & = 
     \frac{1}{r!} \int_0^{\infty} t^{r-1} e^{-t} F\left(e^{-t} z\right) dt. 
\end{align}
Another source of generating function transformation identities correspond to the 
bilateral series given by Lindel\"of in \cite[\S 2]{MILGRAM} of the form 
\begin{equation}
\sum_{n=-\infty}^{\infty} f(n) z^n = -\frac{1}{2\pi\imath} 
     \oint_{\gamma} \pi \cot(\pi w) f(w) z^{w} dw, 
\end{equation}
where $\gamma$ is any closed contour in $\mathbb{C}$ which contains all of the 
singular points of $f$ in its interior. 
In this note, we will focus on integral formulas for generating function 
transformations of an arbitrary sequence, $\{f_n\}_{n \geq 0}$. 

Additional series transformations involving a sequence generating function into the form of 
$\sum_{n \geq 0} f_n z^n / g(n)^s$ 
for $\Re(s) > 1$ and non-zero sequences $\{g(n)\}_{n \geq 0}$ are proved in 
\cite{GFTRANS2016,GFTRANS2018}. 
Note that the harmonic-number-related coefficients implicit to these series 
transformations satisfy summation formulas which are readily expressed by 
N\"orlund-Rice contour integral formulas as well. 
The author has proved in \cite{SQSERIES-MDS} so-called \emph{square series} transformations providing that 
\begin{equation}
\sum_{n \geq 0} f_n q^{n^2} z^n = \frac{1}{\sqrt{2\pi}} \int_0^{\infty} 
     \left[\sum_{b = \pm 1} F\left(e^{bt \sqrt{2\Log(q)}}\right)\right] e^{-t^2 / 2} dt,\ |q|, |qz| < 1. 
\end{equation}
Applications of these square series integral representations include many new 
integral formulas for theta functions and classical $q$-series identities such as the 
\emph{Jacobi triple product} and the partition function generating function, 
$(q; q)_{\infty}^{-1}$, expanded by Euler's \emph{pentagonal number theorem}. 

There are more general \emph{Meinardus methods} for computing asymptotics of the coefficients 
of classes of partition number generating functions of the form \cite{MEINARDUS-METHOD} 
\begin{equation}
\sum_{n \geq 0} p_n(b) z^n := \prod_{k \geq 1} \left(1-z^k\right)^{-b_k}, 
\end{equation}
where $p_n(b)$ denotes the number of weighted partitions of $n$ corresponding to the parameter 
weights $b_k$ for $k \geq 1$. 
Generating functions enumerating partition function sequences of this type are related to a known 
\emph{Euler transform} of a sequence $\{a_n\}_{n \geq 1}$ given by  \cite{OEIS-INTEGERTF}
\begin{equation}
1 + \sum_{n \geq 1} b_n z^n := \prod_{j \geq 1} \frac{1}{\left(1-z^j\right)^{a_j}} 
     \implies 
     \log\left(1 + B(z)\right) = \sum_{k \geq 1} \frac{A(z^k)}{k}, 
\end{equation} 
where $A(z) := \sum_n a_n z^n$ and $B(z) := \sum_n b_n z^n$ are the respective OGFs of the 
component sequences. 
In this case the right-hand-side generating function in the last equation is 
generated succinctly by a $q$-integral for the $q$-beta function of the form 
\cite{ANDREWS-QSERIES} 
\[
\frac{1}{1-q} \int_0^1 f(x) d(z, x) = \sum_{i \geq 0} f(z^i) z^i, 
\]
where inputting the modified generating function, 
$\widetilde{A}_z(t) := A(t) \log(z) / (t \log t)$ for fixed $z$, 
into this integral formula generates the second to last series result. 

\subsection{Results proved in this note}

In this short note we provide proofs of known integral formulas providing an 
ordinary-to-exponential generating function operation. We prove the following 
theorem using the Hankel loop contour for the reciprocal gamma function in 
Section \ref{Section_HankelLoop}. 

\begin{theorem}[OGF-to-EGF Integral Formula I]
\label{theorem_OGF2EGF_iformula_v1} 
For any real $c>0$, provided that $F(z)$ is analytic for $0 < |z| \leq c$, 
we have that 
\[
\widehat{F}(z) =  \sum_{n \geq 0} f_n z^n \int_{-\infty}^{\infty} 
     \frac{e^{c+\imath t}}{(c+\imath t)^{n+1}} dt = 
     \int_{-\infty}^{\infty} 
     \frac{e^{c+\imath t}}{(c+\imath t)} F\left(\frac{z}{c+\imath t}\right) dt. 
\]
\end{theorem} 

\noindent 
We also give a rigorous proof of the next integral formula relating 
$F(z)$ and $\widehat{F}(z)$. 

\begin{theorem}[OGF-to-EGF Integral Formula II]
\label{theorem_OGF2EGF_iformula_v2} 
If $F(z)$ is analytic for $0 < |z| < \sigma_f$, 
we have that \eqref{eqn_CMATH_OGF2EGF_int_formula} holds. 
Namely, we have that 
\[
\widehat{F}(z) = \frac{1}{2\pi} \int_{-\pi}^{\pi} F\left(z e^{\imath t}\right) 
     e^{e^{\imath t}} dt. 
\]
\end{theorem} 
\noindent 
The proof of Theorem \ref{theorem_OGF2EGF_iformula_v2} 
is given in Section \ref{Section_IFormulas_Fourier}. 

\section{Integral representations of the reciprocal gamma function} 
\label{Section_HankelLoop} 

Since $\Gamma(z)$ is a meromorphic function of $z$ with poles at the non-positive integers, it follows 
that the \emph{reciprocal gamma function}, $1 / \Gamma(z)$, is an entire function (of order one) with 
zeros at $z = 0, -1, -2, \ldots$ \cite[\S 5.1]{NISTHB}. 
Indeed, as $|z| \rightarrow \infty$ at a constant $|\operatorname{arg}(z)| < \pi$, we can expand 
\begin{equation}
\log\left[\frac{1}{\Gamma(z)}\right] \sim - z\log z + z + \frac{1}{2}\log\left(\frac{z}{2\pi}\right) - 
     \frac{1}{12z} + \frac{1}{360 z^3} - \frac{1}{1260 z^5}, 
\end{equation}
which can be computed via the infinite products 
\begin{align*} 
\frac{1}{\Gamma(z)} & = z \prod_{n \geq 1} \frac{\left(1 + \frac{z}{n}\right)}{\left(1 + \frac{1}{n}\right)^{z}} 
     = z e^{\gamma z} \prod_{n \geq 1} \left(1 + \frac{z}{n}\right) e^{-z/n}, 
\end{align*} 
where $\gamma \approx 0.577216$ is \emph{Euler's gamma constant}.
Classically, Karl Weierstrass called the function $1 / \Gamma(z)$ the ``\emph{factorielle}'' function, 
and used its representation to prove his famous \emph{Weierstrass factorization theorem} in 
complex analysis \cite[\S 2]{REMMERT-CMPLX-FUNC-THEORY}. 

For $z \in \mathbb{C}$ such that $\Re(z) > 0$ we have a known series expansion for the 
reciprocal gamma function given by 
\begin{equation}
\frac{1}{\Gamma(z)} = \sum_{k=1}^{\infty} a_k z^k = z + \gamma z^2 + \left( 
     \frac{\gamma^2}{2}-\frac{\pi^2}{12}\right) z^3 + 
     \left(\frac{\gamma^3}{6}-\frac{\gamma \pi^2}{12} + \frac{\zeta(3)}{3} 
     \right) z^4 + \cdots. 
\end{equation}
The coefficients $a_k$ in this expansion satisfy many known recurrence relations and 
expansions by the \emph{Riemann zeta function}.  
In \cite{RECIPGAMMA-COEFFS} an exact integral formula for these 
coefficients is given by 
\[
a_n = \frac{(-1)^n}{\pi \cdot n!} \int_0^{\infty} e^{-t} \Im\left\{ 
     \left(\log t-\imath\pi\right)^n\right\} dt. 
\] 
This integral formula is obtained in the reference using Euler's 
reflection formula for the gamma function given by 
\[
\frac{1}{\Gamma(z)} = \frac{\sin(\pi z)}{\pi} \Gamma(1-z), 
\] 
and then applying a standard known real integral to express the 
gamma function on the right-hand-side of the previous equation. 
Equivalently, the reflection formula can be stated as 
\[
\frac{1}{\Gamma(1+z) \Gamma(1-z)} = \frac{\sin(\pi z)}{\pi z}. 
\]

\subsection{The Hankel loop contour for the reciprocal gamma function} 

We seek an exact integral representation for the reciprocal gamma function, 
not just an integral formula defining the coefficients of its 
Taylor series expansion about zero in this case. 
To find such a formula 
we must use the \emph{Hankel loop contour} $H_{\delta,\varepsilon}$ 
shown in Figure \ref{figure_Hankel_loop} and consider the contributions of 
each component section of the contour in the limiting cases for 
increasingly small $\delta,\varepsilon \rightarrow 0$. 
We prove Theorem \ref{theorem_OGF2EGF_iformula_v1} 
using the next lemma derived from this contour below. 

\begin{figure}[ht!] 

\centering 
\begin{framed} 
\includegraphics{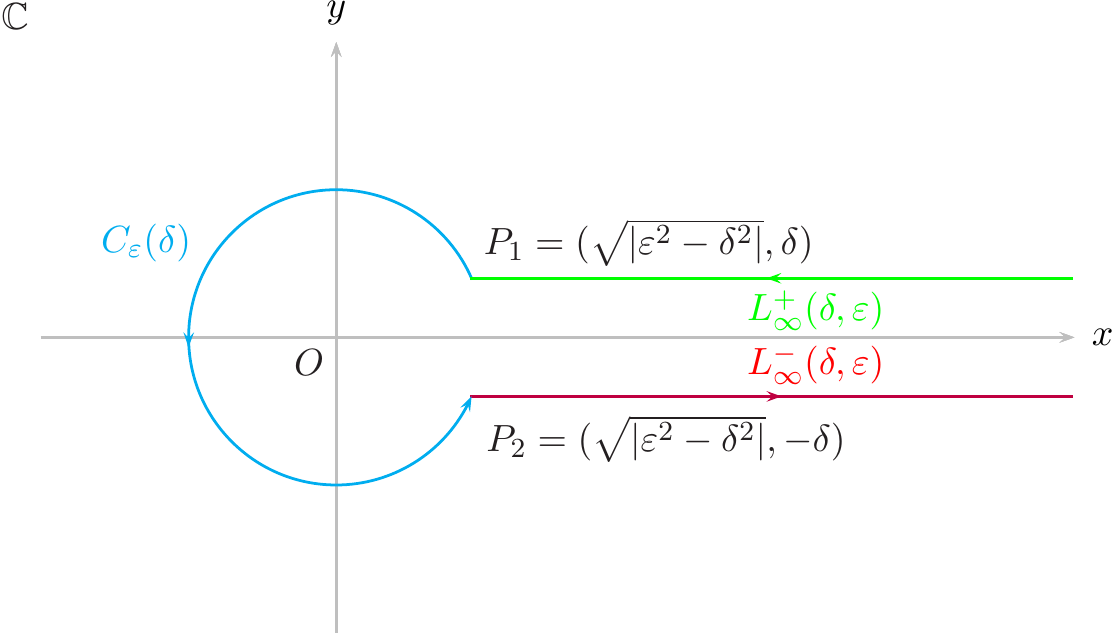}
\end{framed} 

\caption{The Hankel loop contour providing an integral representation of the 
         reciprocal gamma function when $\Re(z) > 0$. } 
\subcaption*{This contour starts positively 
         from the right, traverses the horizontal line $L_{\infty}^{+}(\delta, \varepsilon)$ 
         at distance $+\delta$ from the 
         $x$-axis from $+\infty \rightarrow \sqrt{|\varepsilon^2-\delta^2|}$, then 
         enters the semi-circular loop about the origin of radius $\varepsilon$ 
         denoted by $C_{\varepsilon}(\delta)$ at the point $P_1$, and then at the point 
         $P_2 = (\sqrt{|\varepsilon^2-\delta^2|}, -\delta)$ traverses the last horizontal 
         line $L_{\infty}^{-}(\delta, \varepsilon)$ back to infinity parallel to the 
         $x$-axis. } 
\label{figure_Hankel_loop} 

\end{figure} 

\begin{lemma} 
For any real $c > 0$ and $z \in \mathbb{C}$ such that $\Re(z) > 0$, 
\begin{equation} 
\label{eqn_reciprocal_gamma_func_intformula_v1} 
\frac{1}{\Gamma(z)} = \frac{1}{2\pi} \int_{-\infty}^{\infty} (c+\imath t)^{-z} 
     e^{c+\imath t} dt. 
\end{equation} 
\end{lemma} 
\begin{proof}
\begin{subequations}
Working from the figure, we have that \cite[\S 5.9]{NISTHB} 
\begin{align} 
\frac{1}{\Gamma(z)} & = \lim_{d,\varepsilon\rightarrow 0} 
     \frac{1}{2\pi\imath} \oint_{H_{H_{\delta,\varepsilon}}} (-t)^{-z} e^{-t} dt \\ 
     & = \lim_{d,\varepsilon\rightarrow 0} \frac{1}{2\pi\imath} 
     \left[\int_{C_{\varepsilon}(\delta)} + 
     \int_{L_{\infty}^{+}(\delta, \varepsilon)} + 
     \int_{L_{\infty}^{-}(\delta, \varepsilon)}\right]\left( 
     e^{-\imath\pi z} t^{-z} e^{-t}\right) dt. 
\end{align} 
We will first approach the contribution of the section of the contour 
given by $C_{\varepsilon}$ which is a path enclosing the origin along the 
circle of radius $\varepsilon$ centered at $(0, 0)$. This portion of the 
contour is oriented in the positive direction and begins at the point 
$P_1 := (\sqrt{|\varepsilon^2-\delta^2|}, \delta)$ and ends at the point 
$P_2 := (\sqrt{|\varepsilon^2-\delta^2|}, -\delta)$. 
By parameterizing $t$ along this circle, we obtain the real integral giving 
\begin{align} 
I_{C} & := \lim_{d,\varepsilon\rightarrow 0} \int_{\sin^{-1}\left(-
     \frac{\delta}{\varepsilon}\right)}^{\sin^{-1}\left( 
     \frac{\delta}{\varepsilon}\right)} \imath \varepsilon^2 
     e^{-\imath\pi z} e^{-2i zt} e^{-e^{2\imath t}} dt = 0, 
\end{align} 
since 
$\sin^{-1}\left(\frac{\delta}{\varepsilon}\right) = \frac{\delta}{\varepsilon} 
 + \frac{\delta^3}{6 \varepsilon^3} + O\left(\frac{\delta^5}{\varepsilon^5} 
 \right) \rightarrow 0$ as $\delta,\varepsilon$ independently tend to zero. 
Now we can easily parameterize each of the sections of the contour on the 
horizontal lines each at distance $\delta$ from the $x$-axis. 
In particular, let's define our integrand in the complex parameters $z,w$ as 
$f_{\Gamma}(z, w) := e^{-\imath\pi z} w^{-z} e^{-w}$. Then we consider the 
limiting cases of the following parameterizations of the two line segments 
$\{(s, \pm \delta) : s \in [\sqrt{|\varepsilon^2-\delta^2|}, T]\}$ on 
$L_{\infty}^{+}(\delta, \varepsilon)$ and $L_{\infty}^{-}(\delta, \varepsilon)$, 
respectively, by evaluating the limit of $\delta,\varepsilon \rightarrow 0$ and then 
letting $T$ tend to $+\infty$: 
\begin{align} 
z_{\pm}(\delta, \varepsilon; t) & := \sqrt{|\varepsilon^2-\delta^2|} \pm 
     \imath\delta + t\left( 
     T - \sqrt{|\varepsilon^2-\delta^2|}\right) \\ 
     z_{\pm}^{\prime}(\delta, \varepsilon; t) & \phantom{:} = 
     T - \sqrt{|\varepsilon^2-\delta^2|},\ \mathrm{\ for\ } 
     t \in [0,1]. 
\end{align} 
When we take the first small-order limits we obtain 
\begin{equation}
\lim_{\delta,\varepsilon\rightarrow 0} \int_{0}^{1} f_{\Gamma}\left( 
     z_{\pm}(\delta, \varepsilon; t)\right) \cdot 
     z_{\pm}^{\prime}(\delta, \varepsilon; t) dt = 
     \frac{1}{2\pi\imath} \int_0^T e^{-\imath\pi z} s^{-z} e^{-s} ds, 
\end{equation} 
which by substitution provides us with the symmetric bounds of integration given by 
\begin{equation} 
\lim_{T \rightarrow \infty} 
     \frac{1}{2\pi\imath} \int_0^T e^{-\imath\pi z} s^{-z} e^{-s} ds = 
     \pm \int_{\mp\infty}^0 s^{-z} e^s ds. 
\end{equation} 
We then finally arrive at the stated known integral formula for the 
reciprocal gamma function which holds for any fixed real $c > 0$. 
\end{subequations}
\end{proof} 

\begin{proof}[Proof of Theorem \ref{theorem_OGF2EGF_iformula_v1}] 
Since we are initially motivated by finding a general conversion integral from a 
sequence OGF into its EGF, we notice that we require an application of 
\eqref{eqn_reciprocal_gamma_func_intformula_v1} 
termwise to the Taylor series expansions of our prescribed 
generating function by setting $z = n+1$. 
For example, if we assume that our sequence OGF at hand is well enough behaved when 
its argument satisfies $0 < \Re(z) < c$ for some fixed choice of the real $c > 0$ in the 
integral formula from above, we can sum the integrand of 
\eqref{eqn_reciprocal_gamma_func_intformula_v1} termwise to obtain 
\[
\widehat{F}(z) = \sum_{n \geq 0} f_n z^n \int_{-\infty}^{\infty} 
     \frac{e^{c+\imath t}}{(c+\imath t)^{n+1}} dt = 
     \int_{-\infty}^{\infty} 
     \frac{e^{c+\imath t}}{(c+\imath t)} F\left(\frac{z}{c+\imath t}\right) dt. 
     \qedhere
\]
\end{proof} 

\subsection{Examples: Applications of the integral formula on the real line}  
We can perform the same ``trick'' of the generating function trades to sum a 
``\emph{doubly exponential}'' sequence generating function when we replace the 
sequence OGF by its EGF in the previous equation: 
\begin{equation}
\sum_{n \geq 0} \frac{f_n z^n}{(n!)^2} = 
     \int_{-\infty}^{\infty} 
     \frac{e^{c+\imath t}}{(c+\imath t)} \widehat{F}\left(\frac{z}{c+\imath t}\right) dt. 
\end{equation} 
Perhaps at first glance this iterated integral formula is 
somewhat unsatisfying since we have really just repeated the procedure for 
constructing the first integral twice, but in fact there are notable 
special case applications which we can derive from this method of summation which 
provide new integral representations for otherwise hard-to-sum hypergeometric series. 

For example, if we take the geometric series sequence case where $f_n \equiv 1$ for all 
$n \geq 0$, then we can arrive at a new integral formula for the 
doubly exponential series expansion of the \emph{incomplete Bessel function}, 
$I_0(2\sqrt{z}) = \sum_{n \geq 0} z^n / (n!)^2$ \citep[\S 5.5]{GKP}. 
In particular, we easily obtain that 
\begin{equation}
I_0(2\sqrt{z}) = \int_{-\infty}^{\infty} \frac{e^{c+\imath t}}{c+\imath t} 
     \exp\left(\frac{z}{c+\imath t}\right) dt. 
\end{equation}
There is an integral representation for this function which is simpler to evaluate 
in the general case 
given in \eqref{eqn_CMATH_OGF2EGF_int_formula}. 
We elaborate more on this identity, its proof, and the 
corresponding series involving Stirling numbers which it implies in the next section.

\section{An integral formula from Fourier analysis} 
\label{Section_IFormulas_Fourier} 

One curious identity that the author has come across relating the OGF of a sequence to 
its EGF is found in the appendices of the \emph{Concrete Mathematics} reference 
\cite[p.\ 566]{GKP}. It states \eqref{eqn_CMATH_OGF2EGF_int_formula} 
without proof, again providing that 
\begin{equation*} 
\widehat{F}(z) = \frac{1}{2\pi} \int_{-\pi}^{\pi} F\left(z e^{-\imath t}\right) 
     e^{e^{\imath t}} dt. 
\end{equation*} 
Finding a precise method of verifying this unproven identity is the initial motivation 
for this note. Given the discussion and lead up to an integral for the 
reciprocal gamma function taken over the real line via the Hankel loop contour in the 
last section, the author initially assumed -- and asked with no replies in online 
math forums -- that this computationally correct integral representation must 
correspond to the non-zero components of some complex contour integral. 
It turns out that this formula follows from the basic theory and constructions of 
Fourier analysis. 

\begin{proof}[Proof of Theorem \ref{theorem_OGF2EGF_iformula_v2}] 
Given a sequence, $\{f_n\}_{n \geq 0}$, its (mostly convergent) 
Fourier series is given by $f(x) = \sum_{n \geq 0} f_n e^{\imath\pi n}$. The terms of 
this sequence are then generated by this Fourier series according to the standard 
integral formula \cite{FOURIER-SERIES,INTTFSANDAPPS} 
\[
f_m = \frac{1}{2\pi} \int_{-\pi}^{\pi} f(x) e^{-\imath mx} dx, 
\]
for natural numbers $m \geq 0$. 
If we can assume that the Fourier series, $f(x)$, or equivalently the OGF, 
$F(e^{\imath x})$, is absolutely convergent for all $x \in [-\pi, \pi]$ then we can 
sum over the integral formula in the previous equation to obtain the first 
key component to this proof: 
\[
\sum_{m \geq 0} \frac{f_m z^m}{m!} = \frac{1}{2\pi} \int_{-\pi}^{\pi} 
     F(e^{\imath x}) e^{z e^{-\imath x}} dx. 
\] 
The change of variables $e^{\imath x} = z \cdot e^{-\imath t}$ for fixed $z$
shows that this formula is equivalent to the integral 
formula in \eqref{eqn_CMATH_OGF2EGF_int_formula} directly by a change of variables. 
Also, by expanding the integrand in powers of $e^{\pm\imath x}$ where 
$$\int_{-\pi}^{\pi} e^{\imath(n-k)x} dx = 2\pi \cdot \delta_{n,k},$$ 
it is apparent that these two formulas in fact generate the same power series 
representation for $\widehat{F}(z)$. 
\end{proof}
\begin{proof}[Alternate Proof of Theorem \ref{theorem_OGF2EGF_iformula_v2}]
Another satisfyingly less analytical and more formally motivated explanation for this 
behavior can be given by considering known integral formulas for the 
\emph{Hadamard product} of two series given in terms of the 
orthogonal set $\{e^{\imath kx}\}_{k=-\infty}^{\infty}$ for 
$x$ on the symmetric interval $[-\pi, \pi]$ 
\cite[\S1.12(V); Ex. 1.30, p.\ 85]{ADVCOMB} \cite[\cf \S 6.3]{ECV2}. 
This perspective on the formulations of these two series allows us to swap the 
series variables $z e^{\pm\imath x} \mapsto e^{\pm\imath x}$ from the input of 
one function in the product to another and similarly in the reverse direction. 
Thus we can effectively pick and choose where we would like to position the 
generating function parameter $z$ in each component of the integrand -- 
whether it be situated more naturally as an argument to $F$ as in 
\eqref{eqn_CMATH_OGF2EGF_int_formula}, or whether we choose to keep it 
nested in the corresponding 
multiplier function as in the previous equation. 
We shall see other examples of these integral formula variants in the next remark  
and following examples. 
\end{proof} 

\begin{unremark}[Generalizations of Series Expansions From Fourier Series] 
This technique of using a convergent Fourier series and the corresponding 
integral operation for extracting its coefficients can be generalized to generate 
many other series variants. 
For example, there are many zeta function and polylogarithm-related series 
which are summed by modifying a polylogarithmic series of the form 
expanded in Section \ref{Section_StdIntFormulas_Constructions} by the reciprocal of the 
\emph{central binomial coefficients}, $\binom{2n}{n}$. 
In particular, in the exponential-series-based generating function cases we have that 
\begin{align} 
\sum_{n \geq 0} \frac{f_n z^n}{n! \cdot \binom{2n}{n}} & = 
     \frac{2}{\pi} \int_{-\pi}^{\pi} 
     F(e^{-\imath x}) 
     \frac{\left[\sqrt{4-z e^{\imath x}} + \sqrt{z e^{\imath x}} 
     \sin^{-1}\left(\frac{\sqrt{z e^{\imath x}}}{2}\right) 
     \right]}{(4-z e^{\imath x})^{3/2}} dx \\ 
\notag 
     & = 
     \frac{2}{\pi} \int_{-\pi}^{\pi} 
     F(z e^{-\imath x}) 
     \frac{\left[\sqrt{4-e^{\imath x}} + \sqrt{e^{\imath x}} 
     \sin^{-1}\left(\frac{\sqrt{e^{\imath x}}}{2}\right) 
     \right]}{(4-e^{\imath x})^{3/2}} dx, 
\end{align} 
and in the geometric-series-based OGF cases we recover the exponential error 
function by 
\begin{align} 
\sum_{n \geq 0} \frac{f_n z^n}{\binom{2n}{n}} & = 
     \frac{1}{4\pi} \int_{-\pi}^{\pi} F(e^{-\imath x}) 
     \left[2 + e^{\frac{z e^{\imath x}}{4}} 
     \sqrt{\pi z e^{\imath x}} \erf\left(\frac{\sqrt{z e^{\imath x}}}{2} 
     \right)\right] dx \\ 
\notag
     & = 
     \frac{1}{4\pi} \int_{-\pi}^{\pi} F(z e^{-\imath x}) 
     \left[2 + e^{\frac{e^{\imath x}}{4}} 
     \sqrt{\pi e^{\imath x}} \erf\left(\frac{\sqrt{e^{\imath x}}}{2} 
     \right)\right] dx. 
\end{align} 
There are many other possibilities for constructing integral transformations for 
modified generating function types. 
All one needs to do is be creative and consult a detailed reference of 
compendia such as \cite{II,INTSERIES-TABLES}. 
\end{unremark} 

\subsection{Examples: Generalizations and solutions to a long-standing forum post} 
The primary goal of the 
\href{https://math.stackexchange.com/questions/2274972/an-integral-formula-for-the-reciprocal-gamma-function}{first post} 
mentioned in the introduction was to eventually 
generalize the integral formula in \eqref{eqn_CMATH_OGF2EGF_int_formula} 
to enumerate the modified EGF sequences of the form 
\[
\widehat{F}_{a,b}(z) := \sum_{n \geq 0} \frac{f_n z^n}{\Gamma(an+b+1)}, 
\]
for integers $a \geq 1$ and $b \geq 0$, or over factors of the 
generalized integer multifactorials defined in \cite{MULTIFACTJIS} as 
\[
\ddot{F}_{a,d}(z) := \sum_{n \geq 0} \frac{f_n z^n}{(an+d)!_{(a)}}. 
\]
In the spirit of our realization that the 
integral representation in 
\eqref{eqn_CMATH_OGF2EGF_int_formula} is derived from a Fourier series coefficient formula, 
we may similarly complete our initial goal to sum the second forms of these 
series in the special cases where $(a, b) = (2,0), (2,1)$. 
In particular, we can sum these cases of the 
modified EGFs defined above in closed-form as explicit integral formulas in the forms 
\begin{align} 
\ddot{F}_{2,0}(z) & = \frac{1}{2\pi} \int_{-\pi}^{\pi} F\left(z e^{-\imath t}\right) 
     e^{\frac{1}{2} e^{\imath t}} dt \\ 
\notag 
\ddot{F}_{2,1}(z) & = \frac{1}{2\pi} \int_{-\pi}^{\pi} F\left(z e^{-\imath t}\right) 
     e^{\frac{1}{2} \left[e^{\imath t}-\imath t\right]} 
     \erf\left(\sqrt{\frac{e^{\imath t}}{2}}\right) dt. 
\end{align} 
The modified exponential series of the first type identified above are primarily 
summed in closed-form using expansions of the \emph{Mittag-Leffler functions}, 
$E_{a,b}(z) := \sum_{n \geq 0} z^n / \Gamma(an+b)$, and 
powers of primitive $a^{th}$ roots of unity \cite[\S 10.46]{NISTHB}. 
For example, let's take $(a, b) := (3, 0)$ and observe that 
\begin{equation}
E_{3,0}(t) = \sum_{m \geq 0} \frac{t^m}{\Gamma(3m+1)} = 
     \frac{e^{t^{1/3}}}{3} + \frac{2 e^{-t^{1/3} / 2}}{3} 
     \cos\left(\frac{\sqrt{3} t^{1/3}}{2}\right). 
\end{equation} 
Then we arrive at a corresponding explicit integral representation for the 
modified EGF of any sequence of the form 
\[
\widehat{F}_{3,0}(z) = \frac{1}{2\pi} \int_{-\pi}^{\pi} F(z e^{-\imath t}) 
     E_{3,0}\left(e^{\imath t}\right) dt. 
\] 

\section{Concluding remarks} 
\label{Section_Concl} 

We have proved two key new forms of integral representations for the reciprocal gamma function 
on the real line. By composition and the uniform convergence of power series for functions 
defined on some disc $|z| < \sigma_f$, these results effectively provide us with OGF-to-EGF 
conversion formulas between the generating functions for some $F(z)$. These 
integral formulas for OGF-to-EGF conversion can be applied termwise, or in analytic estimates of the 
asymptotic growth of the coefficients in the power series expansions of the functions defined by the 
corresponding integral transformation. 

We have provided several examples of motivating cases of our so-termed 
\emph{generating function transformations} by integral-based methods in 
Section \ref{subSection_Intro_IntTF_Examples}. 
The broader applications of these transformation methods to other fields and phrasings of 
problems is certainly possible given a suitable context waiting for a new method from which 
to be approached. We hope that readers come away from this article with a new 
understanding of how useful and sometimes indispensable integral transformation methods are in 
sequence analysis.

\end{document}